\documentclass[a4paper,12pt]{article}
\usepackage{amssymb}
\usepackage{mathrsfs}
\usepackage{amsmath}
\usepackage{amsfonts,amsthm,amssymb}
\usepackage{amsfonts}
\usepackage{graphics}
\usepackage{color}
\usepackage{graphicx}
\usepackage{epstopdf}
\usepackage{subfigure}

\newtheorem{lem}{Lemma}[section]
\newtheorem{thm}[lem]{Theorem}
\newtheorem{cor}[lem]{Corollary}

\newtheorem{problem}[lem]{Problem}
\newtheorem{Proposition}[lem]{Proposition}

\begin{document}

\title{On non-empty cross-intersecting families}
\author{Chao Shi$^1$, Peter Frankl$^{2}$\footnote{E-mail: peter.frankl@gmail.com (P. Frankl)}, Jianguo Qian$^1$\footnote{Corresponding author. E-mail: jgqian@xmu.edu.cn (J.G. Qian)}\\
\small 1. School of Mathematical Sciences, Xiamen University, Xiamen 361005, PR China\\
\small {2. R\'{e}nyi Institute, Budapest, Hungary}}
\date{}
\maketitle
{\small{\bf Abstract.}\quad Let $2^{[n]}$ and $\binom{[n]}{i}$ be the power set and the class of all $i$-subsets of $\{1,2,\cdots,n\}$, respectively. We call two families $\mathscr{A}$ and $\mathscr{B}$ cross-intersecting if $A\cap B\neq \emptyset$ for any $A\in \mathscr{A}$ and $B\in \mathscr{B}$. In this paper we show that, for $n\geq k+l,l\geq r\geq 1,c>0$ and $\mathscr{A}\subseteq \binom{[n]}{k},\mathscr{B}\subseteq \binom{[n]}{l}$,  if $\mathscr{A}$ and $\mathscr{B}$ are cross-intersecting and $\binom{n-r}{l-r}\leq|\mathscr{B}|\leq \binom{n-1}{l-1}$, then
$$|\mathscr{A}|+c|\mathscr{B}|\leq \max\left\{\binom{n}{k}-\binom{n-r}{k}+c\binom{n-r}{l-r},\  \binom{n-1}{k-1}+c\binom{n-1}{l-1}\right\}$$
and the families $\mathscr{A}$ and $\mathscr{B}$ attaining the upper bound are also characterized.  This generalizes the corresponding result of Hilton and Milner for $c=1$ and $r=k=l$, and implies a result of Tokushige and the second author (Theorem \ref{thm 1.3}).
\vskip 0.3cm
\noindent{\bf Keywords:} finite set; cross-intersecting; non-empty family

\section{Introduction}
For a natural number $n$, we write $[n]=\{1,2,\cdots,n\}$ and denote by $2^{[n]}$ the power set of $[n]$. In particular, for integer $i>0$ we denote by $\binom{[n]}{i}$ the collection of all $i$-subsets of $[n]$. Every subset  of $2^{[n]}$ is called a {\it family}. We call a family $\mathscr{A}$ {\it intersecting} if $A\cap B\neq \emptyset$ for any $A, B\in\mathscr{A}$, and call $t$ ($t\geq2$) families $\mathscr{A}_1,\mathscr{A}_2,\cdots, \mathscr{A}_t$ {\it cross-intersecting} if $A_i\cap A_j\neq \emptyset$ for any $A_i\in \mathscr{A}_i$ and $A_j\in \mathscr{A}_j$ with $i\neq j$. For a family $A\in 2^{[n]}$, we define its {\it complement} as usual by $\overline{A}=[n]\setminus A$ and, for a family $\mathscr{A}\subset 2^{[n]}$, we denote $\overline{\mathscr{A}}=\{\overline{A}:A\in\mathscr{A}\}$.

The following theorem, known as Erd\H{o}s-Ko-Rado theorem, is a fundamental result in extremal set theory.
\begin{thm}\label{thm 1.1} (Erd\H{o}s-Ko-Rado,\cite{E}). For two positive integers $n$ and $k$, if $n\geq 2k$ and
$\mathscr{A}\subset \binom{[n]}{k}$ is an intersecting family, then
\begin{align*}
|\mathscr{A}| \leq \binom{n-1}{k-1}.
\end{align*}
\end{thm}
The Erd\H{o}s-Ko-Rado theorem has a large number of variations and generalizations, see \cite{A,F0,Furedi,H1,G,Kwan} for examples. A natural direction is to extend the notion of an intersecting family to a class of cross-intersecting families. Notice that if $\mathscr{A}_1=\mathscr{A}_2=\cdots= \mathscr{A}_t$ ($t\geq 2$), then the families $\mathscr{A}_1,\mathscr{A}_2,\cdots, \mathscr{A}_t$ are cross-intersecting if and only if $\mathscr{A}_1$ is intersecting. In this sense, the notion of cross-intersecting for families is indeed a generalization of that of intersecting for a family. The following result was proved by Hilton, a simple proof was given later by Borg \cite{B}.

\begin{thm}\label{HB} (Hilton, \cite{H1}) Let $n,k$ and $t$ be positive integers with $n\geq 2k$ and $t\geq 2$. If $\mathscr{A}_1,\mathscr{A}_2,\cdots, \mathscr{A}_t\subset \binom{[n]}{k}$ are cross-intersecting families, then
\begin{equation*}
\sum\limits_{i=1}^t|\mathscr{A}_i|\leq
\left\{
\begin{aligned}
&\binom{n}{k},&\  \mbox{if}\ t\leq\frac{n}{k};\\
&t\binom{n-1}{k-1},& \  \mbox{if}\ t\geq\frac{n}{k}.
\end{aligned}
\right.
\end{equation*}
\end{thm}

For $t=2$, lots of variations of Theorem \ref{HB} were also considered in the literature by imposing some particular restrictions on the families, e.g., the Sperner type restriction \cite{Wong}, $r$-intersecting restriction \cite{F1,W} and non-empty restriction \cite{F1,H3,W}. For non-empty restriction,  Hilton and Milner gave the following result:
\begin{thm}\label{HM} (Hilton and Milner, \cite{H3}) Let $n$ and $k$ be two positive integers with $n\geq 2k$ and $\mathscr{A},\mathscr{B}\subseteq \binom{[n]}{k}$. If
$\mathscr{A}$ and $\mathscr{B}$ are non-empty cross-intersecting, then
\begin{align*}
|\mathscr{A}|+|\mathscr{B}|\leq \binom{n}{k}-\binom{n-k}{k}+1.
\end{align*}
\end{thm}

In this paper we focus on non-empty cross-intersecting families. Inspired by Theorem \ref{HM}, we prove the following generalization of it.

\begin{thm}\label{mainthm} Let $n,k,l,r$ be any integers with $n\geq k+l, l\geq r\geq 1$, $c$ be a  positive constant  and $\mathscr{A}\subseteq \binom{[n]}{k},\mathscr{B}\subseteq \binom{[n]}{l}$. If $\mathscr{A}$ and $\mathscr{B}$ are cross-intersecting and $\binom{n-r}{l-r}\leq|\mathscr{B}|\leq \binom{n-1}{l-1}$, then
\begin{equation}\label{main2}
|\mathscr{A}|+c|\mathscr{B}|\leq \max\left\{\binom{n}{k}-\binom{n-r}{k}+c\binom{n-r}{l-r},\  \binom{n-1}{k-1}+c\binom{n-1}{l-1}\right\}
\end{equation}
and the upper bound is attained if and only if one of the following holds:\\
{\bf(i).} \begin{equation}\label{iff}
\binom{n}{k}-\binom{n-r}{k}+c\binom{n-r}{l-r}\geq\binom{n-1}{k-1}+c\binom{n-1}{l-1},
\end{equation}
\hspace{9mm}$n>k+l,\mathscr{A}=\{A\in \binom{[n]}{k}:[r]\cap A\neq \emptyset\},\mathscr{B}=\{B\in\binom{[n]}{l}:[r]\subseteq B\}$;\\
{\bf(ii).} The `$\geq$' in (\ref{iff}) is `$\leq$', $n>k+l$, $\mathscr{A}=\{A\in \binom{[n]}{k}:1\in A\},\mathscr{B}=\{B\in \binom{[n]}{l}:1\in B\}$;\\
{\bf(iii).} $n=k+l,c< 1$,  $\mathscr{B}\subset\binom{[n]}{l}$ with $|\mathscr{B}|=\binom{n-r}{l-r},\mathscr{A}=\binom{[n]}{k}\setminus\overline{\mathscr{B}}$;\\
{\bf(iv).}  $n=k+l,c=1$, $\mathscr{B}\subset\binom{[n]}{l}$ with $\binom{n-r}{l-r}\leq|\mathscr{B}|\leq\binom{n-1}{l-1},\mathscr{A}=\binom{[n]}{k}\setminus\overline{\mathscr{B}}$;\\
{\bf(v).} \ $n=k+l,c> 1$,  $\mathscr{B}\subset\binom{[n]}{l}$ with $|\mathscr{B}|=\binom{n-1}{l-1},\mathscr{A}=\binom{[n]}{k}\setminus\overline{\mathscr{B}}$.
\end{thm}

The following result is a simple consequence of Theorem \ref{mainthm}. It provides another generalization of Theorem \ref{HM} by extending two families to arbitrary number of families and it is also sharpening of Theorem \ref{HB}.
\begin{cor}\label{thm 1.7} Let $n,k$ and $t$ be positive integers with $n\geq 2k$ and $t\geq 2$. If $\mathscr{A}_1,\mathscr{A}_2,\cdots,\mathscr{A}_t$ $\subseteq \binom{[n]}{k}$ are non-empty cross-intersecting families, then
\begin{equation}\label{main}
\sum\limits_{i=1}^t|\mathscr{A}_i|\leq \max\left\{\binom{n}{k}-\binom{n-k}{k}+t-1,\  t\binom{n-1}{k-1}\right\}.
\end{equation}
and the upper bound is sharp.
\end{cor}

\section{Proof of Theorem \ref{mainthm} and Corollary \ref{thm 1.7}}

{\bf Proof of Theorem \ref{mainthm}}

Let $\prec_L$, or $\prec $  for short, be the lexicographic order on $\binom{[n]}{i}$ where $i\in\{1,2,\cdots,n\}$, that is, for any two sets $A,B\in\binom{[n]}{i}$,  $A\prec B$  if and only if  $\min\{a:a\in A\setminus B\}<\min\{b:b\in B\setminus A\}$. For a family $\mathscr{A}\subseteq\binom{[n]}{k}$, let $\mathscr{A}_{L}$ denote the family consisting of the first $|\mathscr{A}|$ $k$-sets in order $\prec$, and call $\mathscr{A}$ $L$-{\it initial} if $\mathscr{A}_L=\mathscr{A}$.

In our forthcoming argument,  the well-known Kruskal-Katona theorem  \cite{K1,K2} will play a key role, an equivalent formulation of which was given in \cite{F3,H2} as follows:

\noindent {\bf Kruskal-Katona theorem.} {\it For $\mathscr{A}\in\binom{[n]}{k}$ and $\mathscr{B}\in\binom{[n]}{l}$, if $\mathscr{A}$ and $\mathscr{B}$ are cross-intersecting then $\mathscr{A}_{L}$ and $\mathscr{B}_{L}$  are cross-intersecting as well. }

For any $i\in\{1,2,\cdots,k\}$, let
$$\mathscr{P}^{(l)}_i=\left\{P\in \binom{[n]}{l}:P\supseteq[i]\right\}\ \ {\rm and}\ \ \mathscr{R}^{(k)}_i=\left\{R\in \binom{[n]}{k}:R\cap[i]\neq\emptyset\right\}.$$
\begin{lem}\label{RP} Let $n,k,l$ be any integers with $n\geq k+l$. For any $i\in\{1,2,\cdots,k\}$, $\mathscr{R}^{(k)}_{i}$ is the largest family that is cross-intersecting with $\mathscr{P}^{(l)}_{i}$ and, vice versa. Moreover,  $\mathscr{R}^{(k)}_{i}$ and  $\mathscr{P}^{(l)}_{i}$ are both $L$-initial.
\end{lem}
\begin{proof} Assume that $A$ is a $k$-set that intersects every $l$-set in $\mathscr{P}^{(l)}_{i}$. Choose an arbitrary $(l-i)$-set $B$ from $\{i+1,i+2,\cdots,n\}\setminus A$ (such $B$ exists since $n\geq k+l$). Then $[i]\cup {B}\in\mathscr{P}^{(l)}_{i}$. Since $A\cap B=\emptyset$ and  $A$ intersects every $l$-set in $\mathscr{P}^{(l)}_{i}$, we must have ${A}\cap [i]\not=\emptyset$. Hence, $A\in \mathscr{R}^{(k)}_{i}$ and thus, $\mathscr{R}^{(k)}_{i}$ is largest. The reverse is analogous. Finally, the last part follows directly from the definitions of  $\mathscr{P}^{(l)}_{i}$ and $\mathscr{R}^{(k)}_{i}$.
\end{proof}

By the Kruskal-Katona theorem, when investigating the maximum of $|\mathscr{A}|+c|\mathscr{B}|$, we may assume that both $\mathscr{A}$ and $\mathscr{B}$ are $L$-initial families. Moreover, for given $\mathscr{B}$,  $\mathscr{A}$ is the largest family that is cross-intersecting with $\mathscr{B}$ and vice versa. That is,
$$\mathscr{A}=\left\{A\in \binom{[n]}{k}: A\cap B\not=\emptyset\ \ {\rm for\ all}\ B\in\mathscr{B}\right\},$$
$$\mathscr{B}=\left\{B\in \binom{[n]}{l}: A\cap B\not=\emptyset\ \ {\rm for\ all}\ A\in\mathscr{A}\right\}.$$

We call such a pair $(\mathscr{A},\mathscr{B})$ a {\it maximal pair}.}  Hence, the condition $|\mathscr{B}|\geq \binom{n-r}{l-r}$ implies $\mathscr{P}^{(l)}_r\subseteq \mathscr{B}$. Define $s$ to be the minimal integer such that $\mathscr{P}^{(l)}_s\ \subseteq \ \mathscr{B}$. Therefore, $1\leq s\leq r$.

In the case that $s=1$, we have $\mathscr{P}^{(l)}_{1}=\{P\in \binom{[n]}{l}:1\in P\}$ and $\mathscr{R}^{(k)}_{1}=\{R\in \binom{[n]}{k}:1\in R\}$. Since $\mathscr{P}^{(l)}_1\subseteq \mathscr{B}$,  $\binom{n-1}{l-1}=|\mathscr{P}^{(l)}_1|\leq|\mathscr{B}|\leq \binom{n-1}{l-1}$. This means that the only possibility is $\mathscr{B}=\mathscr{P}^{(l)}_1$. So by Lemma \ref{RP}, $\mathscr{A}=\mathscr{R}^{(k)}_1$ and, hence, $|\mathscr{A}|+c|\mathscr{B}|=\binom{n-1}{k-1}+c\binom{n-1}{l-1}$. Theorem \ref{mainthm} follows in this case.

From now on we assume that $2\leq s\leq r$. By the minimality of $s$, we have
\begin{equation}\label{P}
\mathscr{P}^{(l)}_s\ \subseteq\ \mathscr{B}\subset \mathscr{P}^{(l)}_{s-1}.
\end{equation}

By Lemma \ref{RP}, $\mathscr{B}\subset \mathscr{P}^{(l)}_{s-1}$ means that $\mathscr{R}^{(k)}_{s-1}$ is cross-intersecting with $\mathscr{B}$. Hence,  $\mathscr{R}^{(k)}_{s-1}\subseteq \mathscr{A}$ since $\mathscr{A}$ is largest. On the other hand,  $\mathscr{P}^{(l)}_s\subseteq \mathscr{B}$ means that $\mathscr{A}$ is cross-intersecting  with $\mathscr{P}^{(l)}_s$ since $\mathscr{A}$ is cross-intersecting with $\mathscr{B}$. So, again by Lemma \ref{RP}, we have $\mathscr{A}\subseteq \mathscr{R}^{(k)}_{s}$. In conclusion, (\ref{P}) implies
\begin{equation}
\mathscr{R}^{(k)}_{s-1}\subseteq \mathscr{A}\subseteq \mathscr{R}^{(k)}_s.
 \end{equation}

Consider the cross-intersecting pair $\mathscr{B}_0:=\mathscr{B}\setminus \mathscr{P}^{(l)}_s$ and $\mathscr{A}_0:=\mathscr{A}\setminus \mathscr{R}^{(k)}_{s-1}$. By (\ref{P}), we have
$B\cap[s]=[s-1]$ for any $B\in \mathscr{B}_0$, and $A\cap[s]=\{s\}$  for any $A\in \mathscr{A}_0$. Let
\begin{equation*}
Y=\binom{[s+1,n]}{l-s+1},\ \  X=\binom{[s+1,n]}{k-1},
\end{equation*}
where $[s+1,n]=\{s+1,s+2,\cdots,n\}$. Define $G_s$ to be the bipartite graph with bipartite sets $X$ and $Y$, in which $PQ$ is an edge if and only if $P\in X$, $Q\in Y$ and $P\cap Q=\emptyset$. It is clear that $G_s$ is  {\it biregular}, that is, the vertices in the same partite set have the same degree.
\begin{lem}\label{bigraph} Let $G$ be a bipartite biregular graph with partite sets $P$ and $Q$, and let $c$ be a positive real constant. Let $P_0\subseteq P$ and $Q_0\subseteq Q$. If $P_0\cup Q_0$ is independent, then $|P_0|+c|Q_0|\leq \max\{|P|,c|Q|\}$. Moreover, if $G$ is connected, then equality is possible only for $P_0\cup Q_0=P$ or $Q$.
\end{lem}
\begin{proof}  For a set $W$ of vertices, we denote by $N[W]$ the neighbourhood of $W$. Since $P_0\cup Q_0$ is independent, we have $N(P_0)\cap Q_0=\emptyset$ and $N(Q_0)\cap P_0=\emptyset$. Further, $|N[P_0]|\geq |P_0||Q|/|P|$ and $|N[Q_0]|\geq |Q_0||P|/|Q|$ since $G$ is biregular. Moreover, if $G$ is connected, then equality holds only if $P_0=P$ or $\emptyset$ and $Q_0=Q$  or $\emptyset$. Hence, if $|P|\geq c|Q|$ then we have
$$|P_0|+c|Q_0|\leq |P_0|+\frac{|P|}{|Q|}|Q_0|\leq |P_0|+\frac{|P|}{|Q|}\frac{|Q|}{|P|}|N(Q_0)|\leq |P|.$$

The discussion for the case that $|P|\leq c|Q|$  is analogous. \end{proof}

Let us first consider the case $n>k+l$. Set $\mathscr{A}_1=\{A\backslash[s]:A\in \mathscr{A}_0\}$ and $\mathscr{B}_1=\{B\backslash[s]:B\in \mathscr{B}_0\}$. Then  for any $A\in \mathscr{A}_1$ and $B\in \mathscr{B}_1$, we have $A\cap B\not=\emptyset$ since $A\cup \{s\}\in \mathscr{A},B\cup [s-1]\in \mathscr{B}$ while $\mathscr{A}$ and $\mathscr{B}$ are cross-intersecting. This means that
$\mathscr{A}_1\cup\mathscr{B}_1$ is independent in $G_s$. Let us note that $G_s$ is connected for $n>k+l$. So by Lemma \ref{bigraph},
\begin{equation}\label{bound}
|\mathscr{A}_0|+c|\mathscr{B}_0|=|\mathscr{A}_1|+c|\mathscr{B}_1|\leq\max\{|X|,c|Y|\}.
\end{equation}
Moreover, for $n>k+l$, equality is possible in (\ref{bound}) only if $\mathscr{A}_0=X,\mathscr{B}_0=\emptyset$ or $\mathscr{A}_0=\emptyset,\mathscr{B}_0=Y$. Consequently, either the maximal pair $(\mathscr{A},\mathscr{B})$ is $(\mathscr{R}^{(k)}_s,\mathscr{P}^{(l)}_s)$ or it is $(\mathscr{R}^{(k)}_{s-1},\mathscr{P}^{(l)}_{s-1})$. Hence, we have
\begin{equation}\label{max}
\max\{|\mathscr{A}|+c|\mathscr{B}|\}=|\mathscr{R}^{(k)}_i|+c|\mathscr{P}^{(l)}_i|.
\end{equation}
for some $i\in\{1,2,\cdots, r\}$.

We claim that the maximum in (\ref{max}) is achieved for $i=1$ or $i=r$. To prove this, it is sufficient to prove that there is no $i$ with $2\leq i<r$ satisfying both
\begin{align}
|\mathscr{R}^{(k)}_i|+c|\mathscr{P}^{(l)}_i|\geq |\mathscr{R}^{(k)}_{i-1}|+c|\mathscr{P}^{(l)}_{i-1}|,\\
|\mathscr{R}^{(k)}_i|+c|\mathscr{P}^{(l)}_i|\geq |\mathscr{R}^{(k)}_{i+1}|+c|\mathscr{P}^{(l)}_{i+1}|,
\end{align}
 Equivalently,
\begin{align*}
&\binom{n-i}{k-1}\geq c\binom{n-i+1}{l-i+1}-c\binom{n-i}{l-i}=c\binom{n-i}{l-i+1},\\
&c\binom{n-i-1}{l-i}\geq \binom{n-i-1}{k-1}.
\end{align*}
Multiplying the two inequalities yields
\begin{equation}
c\binom{n-i}{k-1}\binom{n-i-1}{l-i}\ \geq\ c\binom{n-i}{l-i+1}\binom{n-i-1}{k-1}
\end{equation}
or equivalently,
\begin{equation*}
\binom{n-i}{k-1}/\binom{n-i-1}{k-1}\ \geq\ \binom{n-i}{l-i+1}/\binom{n-i-1}{l-i}.
\end{equation*}
Hence,
\begin{equation}
\frac{1}{n-i-k+1}\geq \frac{1}{l-i+1}.
\end{equation}
 This contradicts the assumption that $n> k+l$. Our claim follows.

Thus we have proved that the only maximal pairs are $(\mathscr{R}^{(k)}_1,\mathscr{P}^{(l)}_1)$ or $(\mathscr{R}^{(k)}_r,\mathscr{P}^{(l)}_r)$. This concludes the proof of (\ref{main2}). The uniqueness for initial families follows as well.

To extend uniqueness to general families, we will apply a result proved independently  by F\"{u}redi, Griggs and  M\"{o}rs. To state it we need a definition. For two integers $i,j$ with $n\geq i+j$ and a family $\mathscr{F}\subset \binom{[n]}{i}$, let us define
$$\mathscr{D}_j(\mathscr{F})=\left\{D\in \binom{[n]}{j}:\exists F\in\mathscr{F},D\cap F=\emptyset\right\}.$$
With this terminology, $\mathscr{A}$ and $\mathscr{B}$ are cross-intersecting if and only if $\mathscr{A}\cap\mathscr{D}_k(\mathscr{B})=\emptyset$ or equivalently $\mathscr{B}\cap\mathscr{D}_l(\mathscr{A})=\emptyset$. They form a maximal pair if and only if $\mathscr{A}=\binom{[n]}{k}\setminus\mathscr{D}_k(\mathscr{B})$ and $\mathscr{B}=\binom{[n]}{l}\setminus\mathscr{D}_l(\mathscr{A})$.
\begin{Proposition}\label{FM} (F\"{u}redi, Griggs \cite{Furedi2}, M\"{o}rs \cite{M}). Suppose that $n>k+l, \mathscr{B}\subset\binom{[n]}{l},|\mathscr{B}|=\binom{n-r}{l-r}$ for some $r$ with $1\leq r\leq l$. Then
\begin{equation}\label{FGM}
|\mathscr{D}_k(\mathscr{B})|\geq \binom{n-r}{k}
\end{equation}
with strict inequality unless for some $R\in\binom{[n]}{r},\mathscr{B}=\{B\in\binom{[n]}{l}:R\subset B\}$.
\end{Proposition}

 We should note that (\ref{FGM}) follows from the Kruskal-Katona theorem, the contribution of \cite{Furedi2} and \cite{M} is the uniqueness part. Actually, they proved analogous results for a much wider range but we only need this special case.

Let us continue with the proof of the uniqueness in the case $n>k+l,|\mathscr{A}|=\binom{n}{k}-\binom{n-r}{k},|\mathscr{B}|=\binom{n-r}{l-r}$. From Proposition \ref{FM} and $|\mathscr{A}|=\binom{n}{k}-|\mathscr{D}_k(\mathscr{B})|$, we infer $|\mathscr{D}_k(\mathscr{B})|=\binom{n-r}{k}$. Hence, for some $R\in \binom{[n]}{r}$, $\mathscr{B}=\{B\in\binom{[n]}{l}:R\subset B\}$ and $\mathscr{A}=\binom{[n]}{k}\setminus\mathscr{D}_k(\mathscr{B})=\{A\in\binom{[n]}{k}:A\cap R\not=\emptyset\}$.

Let us next consider the case $n=k+l$. First note that every $k$-set (resp., $l$-set) $F$ is disjoint to only one $l$-set (resp., $k$-set), that is, its complement $\overline{F}=[n]\setminus F$. Consequently, for a family $\mathscr{B}\subset\binom{[n]}{l}$, $\mathscr{D}_k(\mathscr{B})=\overline{\mathscr{B}}=\{\overline{B}:B\in\mathscr{B}\}$. Hence, for any maximal pair $(\mathscr{A}, \mathscr{B})$, $\mathscr{A}=\binom{[n]}{k}\setminus\overline{\mathscr{B}}$ and $|\mathscr{A}|+|\mathscr{B}|=\binom{n}{k}-\binom{n-r}{k}+\binom{n-r}{l-r}=\binom{n}{k}$ since $n=k+l$. This shows that for $c=1$, $|\mathscr{A}|+|\mathscr{B}|=\binom{n}{k}$ holds if and only if $\mathscr{B}$ is an arbitrary family with $\binom{n-r}{l-r}\leq \mathscr{B}\leq \binom{n-1}{l-1}$ and $\mathscr{A}=\binom{[n]}{k}\setminus\overline{\mathscr{B}}$.

For $c>1$, the maximum in Theorem \ref{mainthm} is $\binom{n-1}{k-1}+c\binom{n-1}{l-1}$. It is realized by any pair $(\mathscr{A}, \mathscr{B})$ with $|\mathscr{B}|=\binom{n-1}{l-1},\mathscr{A}=\binom{[n]}{k} \setminus\overline{\mathscr{B}}$.

For $c<1$, the maximum is $\binom{n}{k}-\binom{n-r}{k}+c\binom{n-r}{l-r}$. To realize it we can choose an arbitrary $\mathscr{B}\subset\binom{[n]}{l}$ satisfying $|\mathscr{B}|=\binom{n-r}{l-r}$ and set $\mathscr{A}=\binom{[n]}{k}\setminus \overline{\mathscr{B}}$. This completes the proof of Theorem \ref{mainthm}.

\noindent{\bf Proof of Corollary \ref{thm 1.7}}

Without loss of generality we assume that $|\mathscr{A}_{1}|\geq|\mathscr{A}_2|\geq\cdots\geq|\mathscr{A}_t|$. For $i\in\{1,2,\cdots,t\}$, write  $\mathscr{B}_i=(\mathscr{A}_i)_{L}$. Then, $\mathscr{B}_1\supseteq \mathscr{B}_2\supseteq\cdots\supseteq\mathscr{B}_t$ and $\sum_{i=1}^t|\mathscr{B}_i|=\sum_{i=1}^t|\mathscr{A}_i|$. Further, by the Kruskal-Katona Theorem,  $\mathscr{B}_1, \mathscr{B}_2,\cdots,\mathscr{B}_t$ are cross-intersecting and, therefore, $\mathscr{B}_i$ is intersecting for $i\geq 2$ as $\mathscr{B}_1\supseteq\mathscr{B}_i$. So by the Erd\H{o}s-Ko-Rado theorem, we have $|\mathscr{B}_2|\leq\binom{n-1}{k-1}$. Further,  $\mathscr{B}_1,\mathscr{B}_2,\cdots,\mathscr{B}_2$ are cross-intersecting too. Hence,
$$\sum\limits_{i=1}^t|\mathscr{A}_i|=\sum\limits_{i=1}^t|\mathscr{B}_i|\leq|\mathscr{B}_1|+(t-1)|\mathscr{B}_2|.$$

In Theorem \ref{mainthm}, setting $\mathscr{A}=\mathscr{B}_1, \mathscr{B}=\mathscr{B}_2,c=t-1$ and $r=k=l$, we obtain
$$\sum\limits_{i=1}^t|\mathscr{A}_i|\leq|\mathscr{B}_1|+(t-1)|\mathscr{B}_2|\leq \max\left\{\binom{n}{k}-\binom{n-k}{k}+t-1,\  t\binom{n-1}{k-1}\right\}$$
and the upper bound is sharp. This completes our proof.

\section{An application of Theorem \ref{mainthm}}

Let us recall a related result.
\begin{thm}\label{thm 1.3} (Frankl and Tokushige, \cite{F2})
Let $\mathscr{A}\subset\binom{[n]}{k}$ and $\mathscr{B}\subset\binom{[n]}{l}$ be non-empty cross-intersecting families with $n\geq k+l$ and $k\geq l$. Then
\begin{equation}\label{FT}
|\mathscr{A}|+|\mathscr{B}|\leq \binom{n}{k}-\binom{n-l}{k}+1.
\end{equation}
\end{thm}
We should mention that this result had found several applications, in particular in \cite{F2} it used to provide a simple proof of the following important result.
\begin{thm}\label{stability} (Hilton-Milner stability Theorem \cite{H3}) Suppose that $\mathscr{F}\subset\binom{[n]}{k}$ is intersecting, $\bigcap_{F\in\mathscr{F}}F=\emptyset$ and $n>2k$. Then
\begin{equation}\label{FMEQ}
|\mathscr{F}|\leq \binom{n-1}{k-1}-\binom{n-k-1}{k-1}+1.
\end{equation}
\end{thm}
Let us derive now (\ref{FT}) from Theorem \ref{mainthm}. Without loss of generality, we may assume that $\mathscr{A}=\mathscr{A}_L,\mathscr{B}=\mathscr{B}_L$, i.e., both families are initial. We distinguish two cases.

\noindent{\bf (a).} $|\mathscr{A}|>\binom{n-1}{k-1}$.

Since $\mathscr{A}$ is initial, $\mathscr{A}\supset\mathscr{P}^{(k)}_1=\mathscr{R}^{(k)}_1$ follows. Now the cross-intersecting property implies $\mathscr{B}\supset\mathscr{P}^{(l)}_1$, in particular $|\mathscr{B}|\leq\binom{n-1}{l-1}$.

Applying Theorem \ref{mainthm} with $c=1$ and $r=l$ yields
\begin{equation}\label{10}
|\mathscr{A}|+|\mathscr{B}|\leq \max\left\{\binom{n}{k}-\binom{n-l}{k}+1,\  \binom{n-1}{k-1}+\binom{n-1}{l-1}\right\}.
\end{equation}

\noindent{\bf (b).} $|\mathscr{A}|\leq\binom{n-1}{k-1}$.

Since $|\mathscr{A}|\geq\binom{n-k}{k-k}=1$, we may apply Theorem \ref{mainthm} with the role of $\mathscr{A}$ and $\mathscr{B}$ interchanged, $r=k,c=1$, for obtaining
\begin{equation}\label{11}
|\mathscr{B}|+|\mathscr{A}|\leq \max\left\{\binom{n}{l}-\binom{n-k}{l}+1,\  \binom{n-1}{l-1}+\binom{n-1}{k-1}\right\}.
\end{equation}

Comparing (\ref{10}) and (\ref{11}) with (\ref{FMEQ}), to conclude the proof we must show the following two inequalities:
\begin{equation}\label{12}
\binom{n-1}{k-1}+\binom{n-1}{l-1}\leq \binom{n}{k}-\binom{n-l}{k}+1,
\end{equation}
\begin{equation}\label{13}
\binom{n}{l}-\binom{n-k}{l}\leq \binom{n}{k}-\binom{n-l}{k}.
\end{equation}
Using the formulae
$$\binom{n-1}{l-1}=\binom{n-2}{l-1}+\binom{n-3}{n-2}+\cdots+\binom{n-l-1}{0}$$
and
$$\binom{n}{k}-\binom{n-l}{k}=\binom{n-1}{k-1}+\cdots+\binom{n-l}{k-1},$$
(\ref{12}) is equivalent to
$$\binom{n-2}{l-1}+\cdots+\binom{n-l}{1}\leq\binom{n-2}{k-1}+\cdots+\binom{n-l}{k-1}.$$
This inequality follows by the termwise comparison
\begin{equation}\label{14}
\binom{n-i}{l+1-i}\leq \binom{n-i}{k-1}, \ 2\leq i\leq l.
\end{equation}
Since for $i\geq 2, l+1-i\leq k-1$ and $(l+1-i)+(k-1)=k+l-i\leq n-i$, (\ref{14}) and thereby (\ref{12}) hold.

To prove (\ref{13}) is not hard either. If $n=k+l$ then we have equality. Let us apply induction on $n$, supposing that (\ref{13}) holds for all triples
$(\widetilde{n},\widetilde{k},\widetilde{l})$ with $\widetilde{n}\geq\widetilde{k}+\widetilde{l},\widetilde{k}\geq\widetilde{l}\geq 1$.
Thus we may use the following three inequalities
$$\binom{n}{l}-\binom{n-k}{l}\leq \binom{n}{k}-\binom{n-l}{k},$$
$$\binom{n-1}{l-1}-\binom{(n-1)-(k-1)}{l-1}\leq \binom{n-1}{k-1}-\binom{(n-l)-(l-1)}{k-l}$$
and
$$\binom{n-1}{l-2}\leq \binom{n-1}{k-2}.$$

Summing them up yields (\ref{13}) and concludes the new proof of Theorem \ref{thm 1.3}.

\section{Remark and open problems}

Let us recall that two families $\mathscr{A},\mathscr{B}$ are called cross-$q$-intersecting if $|A\cap B|\geq q$ for all $A\in\mathscr{A},B\in\mathscr{B}$. The following are two related results concerning cross-$q$-intersecting families.
\begin{thm}\label{FK} (Frankl and Kupavskii, \cite{F1})
Let $\mathscr{A}, \mathscr{B}\subset\binom{[n]}{k}$ be non-empty cross-$q$-intersecting families with $k>q\geq 1$ and $n>2k-q$. Then
\begin{align*}
|\mathscr{A}|+|\mathscr{B}|\leq \binom{n}{k}-\sum\limits_{i=0}^{q-1}\binom{k}{i}\binom{n-k}{k-i}+1.
\end{align*}
\end{thm}
\begin{thm}\label{WZ} (Wang and Zhang, \cite{W})
Let $n\geq 4$, $k,l\geq 2$, $q<\min\{k,l\}$, $n>k+l-q$, $(n,q)\neq(k+l,1),\binom{n}{k}\leq\binom{n}{l}$. Then for any non-empty cross-$q$-intersecting families $\mathscr{A}\subset\binom{[n]}{k}$ and $\mathscr{B}\subset\binom{[n]}{l}$,
\begin{align*}
|\mathscr{A}|+|\mathscr{B}|\leq\binom{n}{k}-\sum\limits_{i=0}^{q-1}\binom{k}{i}\binom{n-k}{l-i}+1.
\end{align*}
\end{thm}

Based on the two theorems above, the following three problems are inspired naturally by Corollary \ref{thm 1.7}:

\begin{problem}\label{c1}
Let $\mathscr{A}_1\subset\binom{[n]}{k_1},\mathscr{A}_2\subset\binom{[n]}{k_2},\cdots,\mathscr{A}_t\subset \binom{[n]}{k_t}$ be non-empty cross-intersecting families with $k_1\geq k_2\geq\cdots\geq k_t$, $n\geq k_1+k_2$ and $t\geq 2$. Is it true that
\begin{align*}
\sum\limits_{i=1}^t|\mathscr{A}_i|\leq\max\left\{\binom{n}{k_1}-\binom{n-k_t}{k_1}+\sum\limits_{i=2}^{t}\binom{n-k_t}{k_i-k_t},\ \sum\limits_{i=1}^t\binom{n-1}{k_i-1}\right\}?
\end{align*}
\end{problem}

 We note that if we set $c=t-1$ in Theorem \ref{mainthm}, then we obtain a positive answer to Problem \ref{c1} for the special case that $k_2=\cdots= k_t$.

\begin{problem}\label{c2}
Let $\mathscr{A}_1,\mathscr{A}_2,\cdots, \mathscr{A}_t\subset \binom{[n]}{k}$ be non-empty cross-$q$-intersecting families with $k>q\geq 1$, $n>2k-q$ and $t\geq 2$. Is it true that
\begin{align*}
\sum\limits_{i=1}^t|\mathscr{A}_i|\leq\max\left\{\binom{n}{k}-\sum\limits_{i=0}^{q-1}\binom{k}{i}\binom{n-k}{k-i}+t-1,\ t\binom{n-q}{k-q}\right\}?
\end{align*}
\end{problem}

\begin{problem}\label{c3}
Let $\mathscr{A}_1\subset\binom{[n]}{k_1},\mathscr{A}_2\subset\binom{[n]}{k_2},\cdots,\mathscr{A}_t\subset \binom{[n]}{k_t}$ be non-empty cross-$q$-intersecting families with $k_1\geq k_2\geq\cdots\geq k_t>q\geq 1$, $n>k_1+k_2-q$ and $t\geq 2$. Is it true that
\begin{align*}
\sum\limits_{i=1}^t|\mathscr{A}_i|\leq\max\left\{\binom{n}{k_1}-\sum\limits_{i=0}^{q-1}\binom{k_t}{i}\binom{n-k_t}{k_1-i}+\sum\limits_{i=2}^{t}\binom{n-k_t}{k_i-k_t} ,\ \sum\limits_{i=1}^t\binom{n-q}{k_i-q}\right\}?
\end{align*}
\end{problem}

We note that a positive answer to Problem \ref{c3} would imply that to Problem \ref{c2} and, hence, to Problem \ref{c1}. Moreover, the upper bound in Problem \ref{c3} is attained by setting  $\mathscr{A}_1=\{A\in \binom{[n]}{k_1}:|A\cap[k_t]|\geq q\}$ and $\mathscr{A}_i=\{A\in \binom{[n]}{k_i}:A\supseteq[k_t]\}$ for $i\in\{2,3,\cdots,t\}$ if
\begin{equation}\label{p3}
\binom{n}{k_1}-\sum\limits_{i=0}^{q-1}\binom{k_t}{i}\binom{n-k_t}{k_1-i}+\sum\limits_{i=2}^{t}\binom{n-k_t}{k_i-k_t}\geq \sum\limits_{i=1}^t\binom{n-q}{k_i-q},
\end{equation}
or setting $\mathscr{A}_i=\{A\in \binom{[n]}{k_i}:A\supseteq[q]\}$ for all $i\in\{1,2,\cdots,t\}$ if the `$\geq$' in (\ref{p3}) is `$\leq$'.

\section{Acknowledgements}
Research partially supported by the National Natural Science Foundation of China [Grant numbers, 11971406] and  the Ministry of Education and Science of the Russian Federation in the framework of MegaGrant  [Grant number, 075-15-2019-1926].


\begin{thebibliography}{15}
\bibitem{A} R. Ahlswede, Levon. H. Khachatrian, The complete intersection theorem for systems of finite sets, Europ J. Combin. 18 (1997) 125-136.

\bibitem{B} P. Borg, A short proof of a cross-intersection theorem of Hilton, Discrete Math. 309 (2009) 4750-4753.

\bibitem{E} P. Erd\H{o}s, C. Ko, R. Rado, Intersection theorems for systems of finite sets, Quart. J. Math. Oxf. 2 (12) (1961) 313-320.

\bibitem{F0} P. Frankl, An Erd\H{o}s-Ko-Rado theorem for direct products, Euro J. Combin. 17 (8) (1996) 727-730.

\bibitem{F1} P. Frankl, A. Kupavskii, Uniform $s$-cross-intersecting families, Combinatorics, Probability \& Computing, 26 (4) (2017) 517-524.

\bibitem{F3} P. Frankl, A. Kupavskii, Sharp results concerning disjoint cross-intersecting families, Europ J. Combin 86 (2020) 103089.

\bibitem{F2} P. Frankl, N. Tokushige, Some best possible inequalities concerning cross-intersecting families, J. Combin. Theory Ser. A 61
(1992) 87-97.

\bibitem{Furedi} Z. F\"{u}redi, Cross-Intersecting families of finite sets, J. Combin. Theory Ser. A 72 (1995) 332-339.

\bibitem{Furedi2} Z. F\"{u}redi, Griggs, Families of finite sets with minimum shadows, Combinatorica, 6 (4) (1986) 355-363.

\bibitem{H1} A.J.W. Hilton, An intersection theorem for a collection of families of subsets of a finite set, J. London Math. Soc. 2 (1977)
369-376.

\bibitem{H2} A.J.W. Hilton, The Erd\H{o}s-Ko-Rado theorem with valency conditions, Unpublished Manuscript, 1976.

\bibitem{H3} A.J.W. Hilton, E.C. Milner, Some intersection theorems for systems of finite sets, Quart. J. Math. Oxf. 2 (18) (1967) 369-384.

\bibitem{K1} G.O.H. Katona, A theorem of finite sets, in: Theory of Graphs, Proc. Colloq. Tihany, Akad\'{e}mai Kiad\'{o}, (1968) 187-207.

\bibitem{G} X.L. Kong, Y.X. Xi, G.N. Ge, Multi-part cross-intersecting families, preprint, http://arxiv.org /abs/1809.08756

\bibitem{K2} J.B. Kruskal, The number of simplices in a complex, in: Math. Opt. Techniques, Univ. of Calif. Press, (1963) 251-278.

\bibitem{Kwan} M. Kwan, B. Sudakov, P. Vieira, Non-trivially intersecting multi-part families, J. Combin. Theory, Ser. A, 156 (2018) 44-60.

\bibitem{M} M. M\"{o}rs, A Generalization of a Theorem of Kruskal, Graphs and Combinatorics, 1 (1985) 167-183.

\bibitem{W} J. Wang, H. Zhang, Nontrivial independent sets of bipartite graphs and cross-intersecting families, J. Combin. Theory, Ser. A 120 (2013) 129-141.

%\bibitem{Wilson} R.M. Wilson, The exact bound in the Erd\H{o}s-Ko-Rado theorem, Combinatorica, 4 (1984) 247-257.

\bibitem{Wong} W.H.W. Wong, E.G. Tay, On Cross-intersecting Sperner Families, preprint, http://arxiv.org /abs/2001.01910

\end{thebibliography}
\end{document}